\documentclass[12pt]{amsart}
\title{Hyperbolicity and model-complete fields}
\author{Micha\l{} Szachniewicz}  
\address{Micha\l{} Szachniewicz, Mathematical Institute, University of Oxford, Andrew Wiles Building,
Radcliffe Observatory Quarter, Woodstock Road, Oxford OX2 6GG, UK}
\email{michal.szachniewicz@maths.ox.ac.uk}
\author{Jinhe Ye}
\address{Jinhe Ye, Mathematical Institute, University of Oxford, Andrew Wiles Building,
Radcliffe Observatory Quarter, Woodstock Road, Oxford OX2 6GG, UK}
\email{jinhe.ye@maths.ox.ac.uk}
\usepackage{amsmath, amssymb, amsthm, enumitem, comment}
\usepackage[Symbol]{upgreek}
\usepackage{tikz-cd}
\usepackage[mathscr]{euscript}
\usepackage{fullpage} 	
\usepackage{hyperref}
\usepackage{lineno}
\usepackage[normalem]{ulem}
\usepackage[all]{xy}
\usepackage{centernot}
\usepackage{xcolor}

\DeclareFontFamily{U}{BOONDOX-calo}{\skewchar\font=45 }
\DeclareFontShape{U}{BOONDOX-calo}{m}{n}{
  <-> s*[1.05] BOONDOX-r-calo}{}
\DeclareFontShape{U}{BOONDOX-calo}{b}{n}{
  <-> s*[1.05] BOONDOX-b-calo}{}
\DeclareMathAlphabet{\mathcalboondox}{U}{BOONDOX-calo}{m}{n}
\SetMathAlphabet{\mathcalboondox}{bold}{U}{BOONDOX-calo}{b}{n}
\DeclareMathAlphabet{\mathbcalboondox}{U}{BOONDOX-calo}{b}{n}

\makeatletter
\newcommand*{\rom}[1]{\expandafter\@slowromancap\romannumeral #1@}
\makeatother

\newcommand{\Gal}{\operatorname{Gal}}
\newcommand{\Spec}{\operatorname{Spec}}
\newcommand{\Hom}{\operatorname{Hom}}
\newcommand{\Hilb}{\operatorname{Hilb}}

\newcommand{\Frac}{\operatorname{Frac}}
\newcommand{\acl}{\operatorname{acl}}

\newcommand{\XF}{\mathrm{XF}}
\newcommand{\Abs}{\mathrm{Abs}}

\newcommand{\ACF}{\mathrm{ACF}}

\newcommand{\QQ}{\mathbb{Q}}
\newcommand{\PP}{\mathbb{P}}
\newcommand{\cO}{\mathcal{O}}
\newcommand{\cX}{\mathcal{X}}
\newcommand{\cY}{\mathcal{Y}}
\newcommand{\cZ}{\mathcal{Z}}
\newcommand{\cI}{\mathcal{I}}

\newcommand{\cU}{\mathcal{U}}
\DeclareMathOperator{\id}{id}
\DeclareMathOperator{\Sch}{Sch}
\DeclareMathOperator{\Set}{Set}
\DeclareMathOperator{\cdiv}{div}

\newtheorem*{claim-star}{Claim}
\newtheorem{theorem}{Theorem}[section] 
\newtheorem{lemma}[theorem]{Lemma}

\newtheorem{prop-def}[theorem]{Proposition-Definition}
\newtheorem{corollary}[theorem]{Corollary}

\newtheorem{fact-eh}[theorem]{Fact(?)}

\newtheorem{proposition}[theorem]{Proposition}
\newtheorem{proposition-eh}[theorem]{Proposition(?)}
\newtheorem*{theorem-star}{Theorem}
\newtheorem*{conjecture-star}{Conjecture}
\newtheorem*{question-star}{Question}
\newtheorem*{lemma-star}{Lemma}

\newtheorem*{main-theorem}{Main Theorem}

\theoremstyle{definition}
\newtheorem{definition}[theorem]{Definition}
\newtheorem{example}[theorem]{Example}

\newtheorem{question}[theorem]{Question}

\newtheorem{remark}[theorem]{Remark}
\theoremstyle{remark}

\newtheorem*{acknowledgment}{Acknowledgments}

\newcommand{\cC}{\mathcal{C}}
\newcommand{\cL}{\mathcal{L}}

\newcommand{\ZZ}{\mathbb{Z}}
\newcommand{\RR}{\mathbb{R}}
\newcommand{\CC}{\mathbb{C}}
\newcommand{\Kfield}{k}
\newcommand{\KK}{k}
\def\indecomposably_hyperbolic{indecomposably hyperbolic}
\DeclareMathOperator{\nc}{nc}
\DeclareMathOperator{\con}{c}

\DeclareMathOperator{\Pic}{Pic}
\DeclareMathOperator{\NS}{NS}

\newcommand{\un}{\underline}

\def\Ind#1#2{#1\setbox0=\hbox{$#1x$}\kern\wd0\hbox to 0pt{\hss$#1\mid$\hss}
\lower.9\ht0\hbox to 0pt{\hss$#1\smile$\hss}\kern\wd0}

\def\notind#1#2{#1\setbox0=\hbox{$#1x$}\kern\wd0
\hbox to 0pt{\mathchardef\nn=12854\hss$#1\nn$\kern1.4\wd0\hss}
\hbox to 0pt{\hss$#1\mid$\hss}\lower.9\ht0 \hbox to 0pt{\hss$#1\smile$\hss}\kern\wd0}

\begin{document}

\begin{abstract}
    We study model-complete fields that avoid a given quasi-project variety $V$. There is a close connection between hyperbolicity of $V$ and the existence of the model companion for the theory of characteristic-zero fields avoiding rational points on $V$. This gives a model theoretic notion of hyperbolicity that we call excludability.

   In particular, we show that if $V$ is a Brody hyperbolic projective variety over $\QQ$ with $V(\QQ) = \varnothing$, then the model companion, called $V\XF$, exists. We also study some model-theoretic properties of $V\XF$. This extends the results for curves by Will Johnson and the second author. 
\end{abstract}

\maketitle
\section{Introduction}

As suggested by Artin, maximal subfields of an algebraic closed field avoiding a given set of points exhibit interesting properties, see~\cite{dig_holes_1,dig_holes_2,dig_holes_3,dig_holes_4,dig_holes_5}. Similarly, one could ask about `maximal' fields avoiding points in certain varieties. More precisely, let $\widetilde{V}$ be a smooth projective variety over a characteristic zero field $\KK_0$ with an open subset $V \subset \widetilde{V}$ such that $V(\KK_0) = \varnothing$. Let $T_V$ be the theory of fields extending $\KK_0$ that do not have points in $V$. A field $\KK \models T_V$ can be thought of as `maximal', if it is \emph{existentially closed}, which means that whenever $X$ is a quasi-projective variety over $\KK$ with $X(l) \neq \varnothing$ for some $\KK \subset l \models T_V$, then $X(\KK) \neq \varnothing$. Equivalently, a field $\KK$ extending $\KK_0$ with $V(\KK) = \varnothing$ is existentially closed, if and only if the following conditions hold:
\begin{itemize}
        \item[$(\mathrm{I}_V)$] if $l/\KK$ is any proper finite extension,
        then $V(l) \ne \varnothing$;
        \item[$(\mathrm{II}_V)$] for any geometrically integral
        variety $X/\KK$, either there is a non-constant rational map $f : X \dashrightarrow V_{\KK}$ over $\KK$ or $X(\KK)$ is Zariski dense in $X$.
\end{itemize}
We present an algebro-geometric proof of this equivalence in Section~\ref{sec:prelim}.

If the class of existentially closed models of $T_V$ can be axiomatized by first-order formulas we say that $T_V$ has a \emph{model companion} and the model companion of $T_V$ is the set of such axioms. 

In \cite{johnson2023curveexcluding} the authors proved that if $V$ is a curve $C$ of genus greater or equal two, then the model companion exists. 
\begin{theorem}~\cite[Theorem 2.1]{johnson2023curveexcluding}
    The theory $T_C$ has a model companion $C\XF$.
\end{theorem}

The theories $C\XF$ provide interesting examples of classes of non-large fields. Recall that a field $k$ is \emph{large} if for any smooth curve $D$ over $k$ with $D(k)\neq \varnothing$, the set $D(k)$ is Zariski dense in $D$. Largeness is a tameness notion introduced by Pop to study regular inverse Galois problems~\cite{pop-embedding}.

 It has long been believed that the concept of largeness holds significance in model theory of fields. Specifically, it is believed that largeness plays a crucial role in simplifying conjectures concerning model theory of fields. For instance, Koenigsmann provided a straightforward proof of the Podewski conjecture for large fields of arbitrary characteristic, a conjecture only known to be true in characteristic $p$ \cite{wagner-minimal-fields}.

The relationship between largeness and model-theoretic tameness was extensively investigated in~\cite{firstpaper} through the introduction of the \'etale open topology. As a notable application, the paper resolved the stable fields conjecture for large fields. Subsequently, these methodologies were extended to encompass large simple fields in~\cite{with-anand}. Further exploration of the \'etale open topology was conducted in~\cite{secondpaper,field-top-1,field-top-2}, revealing a uniform approach for studying the structure of definable sets within large fields. Notably, this approach encompasses essentially all previously known examples of model-theoretically tame fields.

In~\cite{johnson2023curveexcluding}, $C\XF$'s were introduced as the very first examples of non-large fields whose logical theories remain tame, answering questions by Koenigsmann and Macintyre concerning largeness, model-complete fields, and their Galois groups. Moreover, $C\XF$'s exhibit intriguing combinations of properties, notably:
\begin{enumerate}
    \item [(1)]  Model-theoretically, they are $\mathrm{NSOP}_4$ and $\mathrm{TP}_2$.
    \item [(2)]  Some models of $C\XF$ give rise to the first known examples of non-large, Hilbertian fields with a decidable theory.
    \item [(3)] The decidability of the partial theory $C\XF$ is equivalent to the effective Mordell conjecture for $C$.
\end{enumerate}  
The properties exhibited by $C\XF$'s suggest that the precise relationship between largeness and logical tameness is more subtle than initially anticipated, warranting further investigation and attention.

To delve deeper into the relationship between largeness and logical tameness, a promising way is to construct and examine new examples. For instance, in \cite[Section 8]{johnson2023curveexcluding}, the authors pose a question regarding the possible extension of their findings to higher-dimensional varieties $V$ (where $V(\KK_0) = \varnothing$) as opposed to a curve $C$. More precisely:
\begin{question}
    What conditions on a quasi-projective variety $V$ over $\KK_0$ ensure that the theory $T_V$ of fields extending $\KK$ with no points in $V$ has a model companion?
\end{question}
Equivalently, one can form it in the following way:
\begin{question}
    What conditions on a quasi-projective variety $V$ over $\KK_0$ ensure that axioms:
    \begin{itemize}
        \item[$(\mathrm{I}_V)$] if $l/\KK$ is any proper finite extension,
        then $V(l) \ne \varnothing$;
        \item[$(\mathrm{II}_V)$] for any geometrically integral
        variety $X/\KK$, either there is a non-constant rational map $f : X \dashrightarrow V_{\KK}$ over $\KK$ or $X(\KK)$ is Zariski dense in $X$;
    \end{itemize}
    are first-order conditions on a field $\KK$ extending $\KK_0$?
\end{question}

We say $V$ is \textit{excludable} if the above question has a positive answer and denote the resulting model companion by $V\XF$. Note that the axiom $(\mathrm{I}_V)$ is always first-order, see Remark~\ref{remark_first_order_I}. In this paper we provide many examples of higher dimensional excludable varieties and study properties of $V\XF$. More precisely, we find a generalization of the condition ``genus bigger or equal two" for a curve. This leads to the following theorem. 
\begin{main-theorem}[Theorem~\ref{theorem_main_result_of_the_paper}]~\label{theorem_intro_main}
    Let $\widetilde{V}$ be a geometrically integral, geometrically pure, projective variety over $\Kfield_0$. Let $V_0 \subset \widetilde{V}$ be a closed subscheme such that $\widetilde{V}$ is geometrically bounded modulo $V_0$. Let $V = \widetilde{V} \setminus V_0$ and assume that $V(\KK_0) = \varnothing$. Then $V$ is excludable.
\end{main-theorem}
In the case $V_0 = \varnothing$, the assumptions of this theorem are satisfied, if for example $\KK_0 = \QQ$ and $V(\CC)$ is Brody hyperbolic. Also, using \cite{complete_int_have_ample_cotangent} and \cite{Faltings} together with the above, one gets that a `random enough' subvariety of an abelian variety has an excludable open set.
\begin{corollary}[Corollary~\ref{corollary_random_subvarieties_are_excludable}]~\label{theorem_random_subvariety_of_abelian_isexcludable_introduction}
    Let $A$ be an abelian variety over $\KK_0$ with an ample line bundle $\cO_A(1)$. Let $\widetilde{V}$ be a sufficiently general intersection of at least half $\dim A$ many high enough powers of $\cO_A(1)$ defined over $\KK_0$. Then $\widetilde{V}(\KK_0)$ is not Zariski dense in $\widetilde{V}$ and if $V \subset \widetilde{V}$ is any open subscheme such that $V(\KK_0) = \varnothing$, then $V$ is excludable.
\end{corollary}

Let us elaborate on the conditions on the pair $V_0 \subset \widetilde{V}$ appearing in the statement of Theorem~\ref{theorem_intro_main}.
Fix a field $\KK$ and a pair of varieties $X_0 \subset X$ over $\KK$ with $X$ projective and $X_0$ a closed subscheme. We say that:
\begin{itemize}
    \item $X$ is pure, if for every smooth variety $T$ over $\KK$ are rational map $T \dashrightarrow X$ extends to a regular map to $X$. 
    \item $X$ is bounded modulo $X_0$, if for every normal projective variety $Y$ over $\KK$, the scheme $\un{\Hom}_{\KK}(Y,X) \setminus \un{\Hom}_{\KK}(Y,X_0)$ is of finite type over $\KK$.
    \item $X$ is indecomposable (over $\KK$), if for any integral, smooth, projective varieties $X_1, X_2$ over $\KK$, any map $X_1 \times X_2 \to X$ factors through $X_1$ or $X_2$.
\end{itemize}
In the case where $\KK$ is not algebraically closed, we add ``geometrically" to any of these properties, if it holds after base-changing to the algebraic closure. For details on pureness, boundedness modulo a subvariety and other notions of hyperbolicity, see Section~\ref{sec:rigidity} and \cite{vanbommel2019boundedness}, \cite{javanpeykar2020langvojta}. The notion of indecomposability is introduced to capture the property of all non-constant maps to $X$ being rigid. We need it to study properties of $V\XF$ for many examples of $V$, in particular we make use out of the following proposition.
\begin{proposition}[Proposition~\ref{proposition_geometric_indecomposability}]
    Let $X$ be a projective variety over $\Kfield_0$. If $X$ is geometrically indecomposable, then $X_l$ is indecomposable over $l$ for any field extension $\Kfield_0 \subset l$. Moreover, $X$ is geometrically indecomposable, if and only if $X_l$ is.
\end{proposition}
It allows us to prove Lemma~\ref{lemma_amalgamation_of_models_of_VXF_forall} which in turn proves the proposition below. Let us highlight that once $V\XF$ is known to exist as a
first-order theory, its model-theoretic analysis more or less follows from \cite{johnson2023curveexcluding}. The real work of this paper in comparison to \cite{johnson2023curveexcluding} lies in (i)
identifying the analogue of genus $>1$ in higher dimensions, and (most importantly)
(ii) proving that the model-companion $V\XF$ exists under this assumption.
\begin{proposition}[Proposition~\ref{theorem_properties_of_VXF}]~\label{theorem_properties_of_VXF_intro}
    Let $V$ be a quasi-projective excludable variety over $\Kfield_0$. Assume that $V$ has a projective compactification $\widetilde{V}$ which is geometrically integral and geometrically indecomposable. Then the theory $V\XF$ has the following properties.
    \begin{enumerate}
        \item Let $l_1, l_2 \models V\XF$. Then $l_1 \equiv l_2$ if and only if $\Abs_{\KK_0}(l_1) \simeq \Abs_{\KK_0}(l_2)$ as $\KK_0$-fields. Moreover, for any algebraic extension $\KK_0 \subset \KK$ avoiding $V$, there is a model $l \models V\XF$ such that $\Abs_{\KK_0}(l) = \KK$.
        \item $V\XF$ has quantifier elimination if we expand by the predicates
        \begin{equation*}
            \mathrm{Sol}_n(x_0,...,x_n)\leftrightarrow \exists y \, x_0+yx_1+...+y^nx_n=0.    
        \end{equation*}  
        \item Model-theoretic and field-theoretic (relative) algebraic closure coincide in models of $V\XF$. In particular, models of $V\XF$ are algebraically bounded over $k_0$.
        \item $V\XF$ is a geometric theory, i.e., $\acl$ satisfies the exchange property and $\exists^\infty$ is uniformly eliminated.
         \item Every proper finite extension of models of $V\XF$ is a Hilbertian PAC field.
         \item Models of $V\XF$ are Hilbertian. Moreover, the absolute Galois group is $\omega$-free.
         \item In the axiomatization of $V\XF$ one can in fact skip the axiom $(\mathrm{I}_V)$. In other words, fields $\KK$ that avoid $V$ and satisfy $(\mathrm{II}_V)$ are models of $V\XF$.
         \item Model-theoretically, $V\XF$ is NSOP$_4$ and TP$_2$.
    \end{enumerate}
\end{proposition}

It is worth mentioning that in the context of Theorem~\ref{theorem_random_subvariety_of_abelian_isexcludable_introduction} the assumptions of Theorem~\ref{theorem_properties_of_VXF_intro} are satisfied because of the Brotbek and Darondeau theorem \cite{complete_int_have_ample_cotangent}. 

Before we conclude the introduction, let us briefly discuss the relationship between excludability and other notions of hyperbolicity. In the main theorem, we prove that boundedness implies excludability. There is a weaker property than boundedness, called grouplessness that is conjectured to be equivalent to boundedness (for details, see Remark~\ref{remark:weak_Lang_Vojta}). If excludability implies grouplessness, it would give a new intermediate property in the list of conjecturally equivalent hyperbolicity properties of a variety \cite[Conjecture 12.2]{javanpeykar2020langvojta}. Thus, we ask the following.

\begin{question}
    Is there a geometrically integral, smooth projective variety $V$ over $k_0$ (with no rational points in $k_0$) that is excludable, but not groupless?
\end{question}

The article is organized as follows. In Section~\ref{sec:prelim} we present some definitions from logic in algebro-geometric flavour. In Sections~\ref{sec:inter} and \ref{sec:hilb} we survey some results from intersection theory and the theory of Hilbert schemes respectively. In the latter section we prove Proposition~\ref{proposition_geometric_indecomposability}. In Section~\ref{sec:rigidity} we describe relations between various notions of hyperbolicity/rigidity in algebraic geometry. In Sections~\ref{sec:model_comp} and \ref{sec:indecomp} we prove the main theorems of the paper. At last, in Section~\ref{sec:rem_and_quest}, we provide an example in characteristic $p$, of a curve $C$ such that $T_C$ does not admit a model companion. This is the content of Proposition~\ref{proposition_counterexample}. We also pose some questions about possible further directions of studies.

\begin{acknowledgment}
We would like to thank Ehud Hrushovski, Ariyan Javanpeykar, and Will Johnson for their helpful remarks. We would also like to thank Piotr Achinger, George Cooper, Joachim Jelisiejew, Adrian Langer, Benedikt Stock, and Jakub Wiaterek for some discussions and references. Lastly, we would like to thank the anonymous referees for their suggestions on improving the paper.
\end{acknowledgment}

\section{Existential closedness and model companion}\label{sec:prelim}
In this section we present some definitions coming from logic in an algebro-geometric flavour. The results follow easily by generalising \cite[Section 1]{johnson2023curveexcluding}, however, we decided to include them here for the convenience of the reader. Fix a characteristic zero field $\KK_0$. We work with the language consisting of plus, minus, multiplication, and constants from $\KK_0$. Let $V$ be a quasi-projective variety over $\KK_0$. 
\begin{definition}
    Let $T_V$ be a theory in the above-introduced language such that models of $T_V$ are fields $\KK$ containing $\KK_0$ with $V(\KK) = \varnothing$. We denote it by $\KK \models T_V$.
\end{definition}
This is a first-order theory, which means that it can be written as a set of axioms being first-order sentences. Recall that first-order means that one can quantify over elements, but not over subsets. For example, if $V \subset \PP^N$ is the intersection of the zero set of homogeneous polynomials $f_1, \dots, f_n$ and the non-vanishing locus of homogeneous polynomials $g_1, \dots, g_m$, then $T_V$ can be given by the following axioms.
\begin{itemize}
    \item Axioms for relations in $\KK_0$: if for example $a+b=c$ in $\KK_0$, then we put here the corresponding relation among constants from $\KK_0$.
    \item Field axioms: for example $(\forall x)(x+0=0+x=x)$ or $(\forall x)(x \neq 0 \to (\exists y)(x \cdot y = 1))$.
    \item Excluding $V$ axiom: $\neg (\exists x=(x_0, \dots, x_N))(x \neq 0 \wedge \bigwedge_{i=1}^n f_i(x)=0 \wedge \bigvee_{j=1}^m g_j(x) \neq 0)$.
\end{itemize}
Note that this theory is consistent (i.e., has a model) if and only if $V(\KK_0)$ is already empty. Moreover, it is clear that this theory only depends on the reduction of $V$. We now present the definition of existential closedness in an algebro-geometric fashion. 
\begin{definition}
If $\KK$ is a model of $T_V$, then $\KK$ is called \textit{existentially closed} if for every quasi-projective variety $X$ over $\KK$, if there exists an extension $\KK \subset l$ within models of $T_V$ such that $X(l) \neq \varnothing$ (i.e., $X$ has a $\Spec(l)$-point over $\Spec(\KK)$), then $X(\KK)\neq \varnothing$.
\end{definition}
\begin{remark}
    If an embedding $\KK \subset l$ of fields satisfies the condition from the above definition, for every quasi-projective variety $X$ over $\KK$, then we say that $\KK$ is existentially closed in $l$.
\end{remark}

\begin{example}~\label{example_empty_variety_model_companion}
    If $V$ is the empty variety, then by Nullstellensatz, a field is existentially closed if and only if it it algebraically closed. Thus, one can think of existentially closed models of $T_V$ as models of $T_V$ that are as close to algebraically closed fields as the condition of not having points in $V$ allows.
\end{example}
If $T_V$ is consistent, existentially closed models exist, but the question is when do they also form an elementary class.
\begin{definition}
    We say that $T_V$ has a \textit{model companion} if the class of existentially closed models of $T_V$ is \textit{elementary}, i.e., if there exists a set $T$ of first-order axioms such that existentially closed models of $T_V$ are exactly the models of $T_V$ satisfying the axioms from $T$. Note that if such $T$ exists, it is unique up to logical consequences.
\end{definition}
\begin{example}
    If $V$ is the empty variety, then as mentioned in Example~\ref{example_empty_variety_model_companion}, existentially closed models of $T_V$ are just algebraically closed fields and those can be axiomatized by the following countable list of axioms:
    \[ \{(\forall y_0, \dots, y_{n-1})(\exists x)(x^n + y_{n-1}x^{n-1} + \dots + y_1 x + y_0 =0) : n \in \mathbb{N}\}. \]
\end{example}
\begin{definition}
    Let $V$ be a quasi-projective variety over $\KK_0$. We call it \textit{excludable} if the theory $T_V$ has a model companion. Then we denote by $V\XF$ the theory axiomatizing existentially closed models of $T_V$.
\end{definition}

Assume from now on that $V$ is a geometrically integral positive-dimensional variety with $V(\KK_0)=\varnothing$. In the rest of this section we focus on giving conditions for being an existentially closed model of $T_V$. 

\begin{lemma}~\label{lemma_it_is_enough_to_check_ec_on_finitely_generated}
    If $\KK$ is existentially closed in all finitely generated extensions $\KK \subset l$ within $T_V$, then $\KK$ is an existentially closed model of $T_V$.
\end{lemma}
\begin{proof}
    This follows from the fact that a subfield of a model of $T_V$ (containing $\KK_0$) is a model of $T_V$.
\end{proof}

\begin{lemma}~\label{lemma_finite_ext_add_points}
    If $\KK$ is an existentially closed model of $T_V$ and $\KK \subset l$ is a proper finite extension of fields, then $V(l) \neq \varnothing$.
\end{lemma}
\begin{proof}
    If $V(l) = \varnothing$, then $l \models T_V$. Let $a \in l \setminus \KK$ be an element with a minimal polynomial $P(x) \in \KK[x]$. From the definition of existentially closedness, we then get that the equation $P(x)=0$ has a solution in $\KK$, which gives us a contradiction.
\end{proof}
The next lemma is well-known, but we include it for the sake of completeness.
\begin{lemma}~\label{lemma_characterisation_of_geometriacally_irreducible}
    Let $X$ be an integral variety over $\KK$. Then $X$ is geometrically integral if and only if $\KK \subset \KK(X)$ is regular, i.e., there is no proper finite extension $\KK \subset l$ with $l$ contained in $\KK(X)$.
\end{lemma}
\begin{proof}
    Note that by \cite[Lemma 020I]{stacks-project} the scheme $X$ is geometrically integral if and only if it is geometrically irreducible, as we assume characteristic zero. By \cite[Lemma 054Q]{stacks-project} this is equivalent to the field $\KK(X)$ being geometrically irreducible over $\KK$. By \cite[Lemma 038I]{stacks-project} this is equivalent to $\KK(X) \otimes_{\KK} \KK'$ being irreducible for every finite extension $\KK'/\KK$. But by the primitive element theorem, every finite extension of $\KK$ is of the form $\KK[x]/(P)$ for some irreducible polynomial $P \in \KK[x]$. Thus the geometric integrability of $X$ is equivalent to the fact that for an irreducible $P \in \KK[x]$, it is still irreducible as an element of $\KK(X)[x]$. This is equivalent to the extension $\KK \subset \KK(X)$ being regular.
    
\end{proof}

\begin{remark}~\label{remark_rational_maps_from_integral}
    Let $X$ be an integral variety over $\KK$. Then the set of rational maps $X \dashrightarrow V_{\KK}$ over $\KK$ can be naturally identified with the set $V_{\KK}(\KK(X))$ (morphisms in $\Sch_{\KK}$). Under this identification $V_{\KK}(\KK)$ corresponds to constant maps (factoring through $\Spec(\KK))$ from $X$ to $V_{\KK}$.
\end{remark}

\begin{lemma}~\label{lemma_characterisation_of_being_existentially_closed_in_function_field}
    Let $X$ be an integral variety over $\KK$. Then the embedding $\KK \subset \KK(X)$ is existentially closed, if and only if $X(\KK)$ is Zariski dense in $X$.
\end{lemma}
\begin{proof}
    First, assume that $\KK \subset \KK(X)$ is existentially closed. If $U \subset X$ is a Zariski open non-empty subset, then $U(\KK(X)) \neq \varnothing$, so by existential closedness, we get that $U(\KK) \neq \varnothing$. Hence $X(\KK)$ is Zariski dense in $X$.

    Second, assume that $X(\KK)$ is Zariski dense in $X$. Let $Y$ be a quasi-projective variety over $\KK$ and assume that $Y(\KK(X))$ is non-empty. By Remark~\ref{remark_rational_maps_from_integral} we conclude that there is a rational map $f:X \dashrightarrow Y$ over $\KK$. Let $U \subset X$ be an open subset such that $f:U \to Y$ is a regular map. By assumption $U(\KK)$ is non-empty, so we get that $Y$ has a $\KK$-point, which finishes the proof.
\end{proof}

\begin{corollary}
    Let $\KK$ be a model of $T_V$. It is existentially closed if and only if the following conditions hold:
    \begin{itemize}
        \item[$(\mathrm{I}_V)$] If $l/\KK$ is any proper finite extension,
        then $V(l) \ne \varnothing$.
        \item[$(\mathrm{II}_V)$] For any geometrically integral
        variety $X/\KK$, either there is a non-constant rational map $f : X \dashrightarrow V_{\KK}$ over $\KK$ or $X(\KK)$ is Zariski dense in $X$.
    \end{itemize}
\end{corollary}
\begin{proof}
    Assume first that $\KK$ is existentially closed. Point $(\mathrm{I}_V)$ is Lemma~\ref{lemma_finite_ext_add_points}. Fix a geometrically integral variety $X/\KK$ and assume there is no non-constant rational map $f : X \dashrightarrow V_{\KK}$ over $\KK$. By Remark~\ref{remark_rational_maps_from_integral} we have $V_{\KK}(\KK(X)) = V_{\KK}(\KK) = \varnothing$. Thus $\KK(X)$ is a model of $T_V$. By Lemma~\ref{lemma_characterisation_of_being_existentially_closed_in_function_field} we are done.  
    
    On the other hand, assume that $\KK$ satisfies $(\mathrm{I}_V), (\mathrm{II}_V)$. By Lemma~\ref{lemma_it_is_enough_to_check_ec_on_finitely_generated} it is enough to check that $\KK$ is existentially closed in every finitely generated extension $\KK \subset F$ with $F \models T_V$. By $(\mathrm{I}_V)$ there is no proper finite extension $\KK \subset l$ with $l$ contained in $F$. Pick an integral variety $X$ over $\KK$ with $\KK(X) = F$ over $\KK$. By Lemma~\ref{lemma_characterisation_of_geometriacally_irreducible} we know that $X$ is geometrically integral and by Remark~\ref{remark_rational_maps_from_integral} there is no rational map $X \dashrightarrow V_{\KK}$ over $\KK$. Hence, $X(\KK)$ is Zariski dense in $X$ by $(\mathrm{II}_V)$. By Lemma~\ref{lemma_characterisation_of_being_existentially_closed_in_function_field} we are done.
\end{proof}

\begin{remark}~\label{remark_first_order_I}
    The condition $(\mathrm{I}_V)$ above is a first-order condition on a field $\KK$. Indeed, one can present it as an axiom schema of the form ``for all degree $n$ extensions $l/\KK$ there is a point in $V(l)$" for all natural $n$. These are first-order, because a degree $n$ extension of $l/\KK$ can be encoded by a tuple of elements $\alpha_{ijk} \in \KK$ such that if $\{e_i\}_{i=1, \dots, n}$ is a basis of $l$ over $\KK$, then $e_i \cdot e_j = \sum_k \alpha_{ijk} e_k$.

    The condition $(\mathrm{II}_V)$ is a priori not a first-order condition on $\KK$. The main result of this paper is that some hyperbolicity assumptions on $V$ actually make it first-order.
\end{remark}

\section{Intersection theory}~\label{sec:inter}
In this section we present some properties of the intersection product in algebraic geometry. We work over an algebraically closed field $\Kfield$ of characteristic zero. For a more thorough introduction to intersection theory we refer the reader to \cite{Fulton-intersection} and \cite[Chapter 21]{Ravi_Vakil_FOAG}.
\begin{definition}
    For an irreducible projective variety $X$ over $\Kfield$ we denote by $A_i(X)$ the $i$'th \textit{Chow group} of $X$ consisting of cycles of dimension $i$ (i.e., integer combinations of irreducible subvarieties of $X$ of dimension $i$) up to rational equivalence. In particular $A_{d-1}(X)$ is the group of Weil divisors on $X$ modulo rational equivalence, if $d=\dim X$. If $\alpha \in A_e(X)$ and $L_1, \dots, L_e$ are line bundles on $X$, then we can form an intersection number denoted by
    \[ \int_{\alpha} c_1(L_1) \cdot \ldots \cdot c_1(L_e) \in \ZZ. \]
    If $\alpha = [Z]$ is a class of a subvariety $Z \subset X$ we denote this also by
    \[ \int_{Z} c_1(L_1) \cdot \ldots \cdot c_1(L_e). \]
    These intersections numbers do not depend on the order of $L_i$'s and are multilinear with respect to the tensor product of line bundles.
\end{definition}

\begin{remark}
    We will need a few operations/properties from intersection theory. We describe them below.
    \begin{enumerate}
        \item Let $s_1, \ldots, s_p$ be global sections of $L_1, \ldots, L_p$ respectively, and assume that their zero sets restricted to $Z$ form a complete intersection $W$. Recall that this means that $s_i$ is not a zero divisor locally on $Z \cap \cdiv(s_1) \cap \ldots \cap \cdiv(s_{i-1})$ for $i \leq p$ and $W = Z \cap \bigcap_{i \leq p} \cdiv(s_i)$. Then
        \[ \int_{Z} c_1(L_1) \cdot \ldots \cdot c_1(L_e) = \int_{W} c_1(L_{p+1}) \cdot \ldots \cdot c_1(L_e). \]
        Moreover, if $p = e$, then $W$ is just a disjoint sum of copies of $\Spec(\Kfield)$ and the number of copies is exactly the intersection number.
        \item For a morphism $f:X \to Y$ between proper varieties, there is a family of \emph{pushforward} maps $f_* : A_e(X) \to A_e(Y)$ for all $e$. For an irreducible subvariety $V \subset X$ we have
        \[ f_*[V] = \deg(V/f(V))[f(V)], \]
        where by $[-]$ we denote the class of a subvariety in the Chow group, and $\deg(V/f(V))$ is the degree of the extension of the field of rational functions on $f(V)$ by the field of rational functions on $V$. If $\deg(V/f(V))$ is infinite, we put $f_*[V]=0$.
        \item For a morphism $f:X \to Y$ between schemes, if $L$ is a line bundle on $Y$, we can pull it back on $X$ and we denote the \emph{pullback} by $f^*L$. If moreover $X$ and $Y$ are proper irreducible varieties and $\alpha \in A_e(X)$, then the \emph{projection formula} holds:
        \[ \int_{\alpha} c_1(f^* L_1) \cdot \ldots \cdot c_1(f^* L_e) = \int_{f_* \alpha} c_1(L_1) \cdot \ldots \cdot c_1(L_e). \]
        \item The intersection number
        $\int_{Z} c_1(L_1) \cdot \ldots \cdot c_1(L_e)$
        can be defined in terms of Euler characteristics of tensor products of duals of $L_i$'s restricted to $Z$. In particular, if $\cX \to S$ is a morphism, $\cL_1, \ldots, \cL_e$ are line bundles on $\cX$ and we are given a subscheme $\cZ \subset \cX$ flat over $S$ (for the definition of flatness see Remark~\ref{remark_flatness}) with relative dimension $e$, then for any field $l$ and any $s \in S(l)$ the numbers \[ \int_{\cZ_s} c_1((\cL_1)_s) \cdot \ldots \cdot c_1((\cL_e)_s) \]
        are constant. This is because Euler characteristic is constant in flat families.
    \end{enumerate}
\end{remark}

\begin{remark}
    If $X$ is smooth, then one usually writes $A^i(X) := A_{d-i}(X)$ and $A^*(X) := \bigoplus_{i=0}^d A^i(X)$ has a structure of a graded ring. 
\end{remark}

Let $X$ be a projective variety over $\Kfield$. We present the definition of the Néron–Severi group of $X$.
\begin{definition}
    If $L$ is a line bundle on $X$ and $C$ is an $1$-cycle in $X$, then we define $L.C := \int_C c_1(L)$. See \cite[Chapter 21]{Ravi_Vakil_FOAG} for the definition in the non-smooth case. This operation gives a pairing
    \[ \Pic(X) \times A_1(X) \to \ZZ, \]
    where $\Pic(X)$ is the set of isomorphisms classes of line bundles on $X$. Two line bundles $L, M \in \Pic(X)$ are \textit{numerically equivalent} (denoted $L \equiv M$) if for all $\gamma \in A_1(X)$ we have $L.\gamma = M.\gamma$. Similarly, two $1$-cycles $\gamma, \delta \in A_1(X)$ are numerically equivalent (denoted $\gamma \equiv \delta$) if for all line bundles $L \in \Pic(X)$ we have $L.\gamma = L.\delta$. The \textit{Néron–Severi group} of $X$ is the group
    \[ \NS(X) := \Pic(X)/\equiv. \]
    It is a finitely generated abelian group and we use the notation $N^1(X) := \NS(X) \otimes_{\ZZ} \RR$. Moreover, we write $N_1(X) := (A_1(X)/\equiv) \otimes_{\ZZ} \RR$ and intersecting line bundles and $1$-cycles yields a perfect pairing
    \[ N^1(X) \times N_1(X) \to \RR. \]
\end{definition}
Below we introduce some important cones in $N^1(X), N_1(X)$.
\begin{definition}
    \begin{enumerate}
        \item The \textit{effective cone} of curves is the closed convex cone in $N_1(X)$ spanned by classes of irreducible curves $C \subset X$.
        \item The \textit{nef cone} is the closed convex cone in $N^1(X)$ consisting of classes $\alpha \in N^1(X)$ such that for all irreducible curves $C \subset X$ we have $\alpha.[C] \geq 0$. In other words, it is the dual cone to the effective cone of curves. 
        \item The \textit{ample cone} is the interior of the nef cone in $N^1(X)$. Equivalently, it is the open cone generated by the classes of ample line bundles. Recall that $L \in \Pic(X)$ is ample, if there exists a positive tensor power $L^{\otimes n}$ of $L$ which is very ample, i.e., such that the mapping
        \[ X \to \PP^N \]
        \[ x \mapsto [s_0(x):s_1(x):\dots:s_N(x)], \]
        defines a closed embedding, where $s_0, \dots, s_N$ is some basis of $H^0(X, L^{\otimes n})$.
    \end{enumerate}
\end{definition}

\begin{remark}~\label{Khovanskii_Teissier_inequalities}
    Assume that $d$ is the dimension of $X$ and let $L, M$ be nef line bundles on $X$. For $i=0, \dots, d$ define $s_i = \int_X c_1(L)^i c_1(M)^{d-i}$. Then the Khovanskii-Teissier inequalities assert that the sequence $s_i$ is log-concave, i.e.,
    \[ s_i^2 \geq s_{i-1} s_{i+1}, \]
    for $i=1, \dots, d-1$. For a proof, see \cite[Example 1.6.4]{Lazarsfeld_positivity_I}. Note that it follows from the projection formula, that a pullback of a nef line bundle is also nef, hence one can use these inequalities also in such context.
\end{remark}

\section{Hilbert schemes}~\label{sec:hilb}

In this section we introduce Hilbert schemes and mention a few results needed later in the text. See also \cite{fantechi2005fundamental} for a detailed introduction and \cite[Section 6]{martinpizarro2024noetheriantheories} for an introduction avoiding the language of schemes. Let $S$ be a scheme which will serve as our base-scheme. We denote by $\Kfield$ an arbitrary field. If $X$ is a projective variety over $\Kfield$, then the Hilbert scheme $\un{\Hilb}_{\Kfield}(X)$ is a countable sum of projective schemes over $\Kfield$ such that $\Kfield$-points of $\un{\Hilb}_{\Kfield}(X)$ correspond to closed subschemes $Z \subset X$. For a family of varieties $\cX \to S$ indexed by a variety $S$ one can put the Hilbert schemes of the fibers into one scheme over $S$. We introduce it below.

\begin{definition}
    Let $\cX \to S$ be a flat projective morphism with a relatively ample line bundle $\cO(1)$ on $\cX$. The \textit{Hilbert scheme} of $\cX$ over $S$ denoted by $\un{\Hilb}_S(\cX)$ is a scheme over $S$ which is characterised by the following property. For every $T \in \Sch_S$ morphisms from $T$ to $\un{\Hilb}_S(\cX)$ can be described as
    \[ \un{\Hilb}_S(\cX)(T) = \bigl\{ \cZ \subset \cX_T : \cZ \textnormal{ is a closed subscheme of } \cX_T \textnormal{ flat over } T \bigr\}. \]
    There exists a disjoint decomposition
    \[ \un{\Hilb}_S(\cX) = \bigcup_P \un{\Hilb}_S^P(\cX), \]
    where $P$ varies among polynomials in $\QQ[t]$, such that each $\un{\Hilb}_S^P(\cX)$ is a projective scheme over $S$. Moreover, for a field $\Kfield$ and $s \in S(\Kfield)$, the $\Kfield$-points of $\un{\Hilb}_S^P(\cX)$ over $s$ are in a natural bijection with
    \[ \bigl\{ Z \subset \cX_s : Z \textnormal{ is a closed subscheme of } \cX_s \textnormal{ with the Hilbert polynomial } P \bigr\}. \]
    We recall that the \textit{Hilbert polynomial} of a closed subscheme $Z \subset \cX_s$ (with respect to $\cO(1)$) is the unique polynomial $P \in \QQ[t]$ such that for natural numbers $m$ big enough we have
    \[ P(m) = \dim_{\Kfield} H^0(Z, \cO(m)|_Z), \]
    or equivalently $P(m) = \chi(\cX_s, \cI_Z(m))$ for all natural $m$, with $\cI_Z$ being the ideal sheaf of $Z$ in $\cX_s$ and $\chi$ being the Euler characteristic.
\end{definition}

\begin{remark}~\label{remark_flatness}
    For the definition of flatness we refer to \cite{stacks-project}. A projective morphism between varieties over $\Kfield$ is flat if and only if the Hilbert polynomials of all geometric fibers are the same.
\end{remark}

\begin{remark}~\label{remark_nice_properties_hilbert_schemes}
    Fix the notation from the above definition. We list here some useful properties of Hilbert schemes. The first three properties follow from the definition of functor defying Hilbert schemes, see e.g. \cite{fantechi2005fundamental} for details. The last property is \cite[Example 2.5.2]{Fulton-intersection} or \cite[Proposition 0BJ8]{stacks-project}.
    \begin{itemize}
        \item If $\cY$ is a closed subscheme of $\cX$ flat over $S$, then there is a natural closed immersion $\un{\Hilb}_S(\cY) \subset \un{\Hilb}_S(\cX)$.
        \item For $s \in S(\Kfield)$ the base-change of $\un{\Hilb}_S(\cX)$ via the map $s:\Spec(\Kfield) \to S$ is isomorphic to $\un{\Hilb}_{\Kfield}(\cX_s)$.
        \item The identity morphism $\un{\Hilb}_S(\cX) \to \un{\Hilb}_S(\cX)$ corresponds to a canonical closed subscheme $\cU \subset \cX \times_S \un{\Hilb}_S(\cX)$ flat over $\un{\Hilb}_S(\cX)$, which we refer to as the \textit{universal family} of $\un{\Hilb}_S(\cX)$. Consider $f:\Spec(\Kfield) \to \un{\Hilb}_S(\cX)$ corresponding to a closed subscheme $Z \subset \cX_s$ with $s \in S(\Kfield)$ being the composition of $f$ and the projection to $S$. If we form the Cartesian diagram
        \[\begin{tikzcd}
        	{\cU_s} & \cU \\
        	{\cX_s} & {\cX \times_S \un{\Hilb}_S(\cX)} \\
        	{\Spec(\Kfield)} & {\un{\Hilb}_S(\cX)}
        	\arrow[from=2-1, to=2-2]
        	\arrow["f", from=3-1, to=3-2]
        	\arrow[from=2-2, to=3-2]
        	\arrow[from=2-1, to=3-1]
        	\arrow[hook', from=1-2, to=2-2]
        	\arrow[hook', from=1-1, to=2-1]
        	\arrow[from=1-1, to=1-2]
        \end{tikzcd}\]
        the subscheme $Z$ will be exactly the image of $\cU_s$ in $\cX_s$ (the ideal sheaves will be the same).
        \item Let $P$ be the Hilbert polynomial of $[Z] \in \un{\Hilb}_{\Kfield}(X)(\Kfield)$ for a projective variety $X$ over algebraically an closed field $\Kfield$ with respect to an ample line bundle $L$ on $X$. Then the degree of $P$ is $d := \dim Z$ and the leading coefficient of $P$ can be calculated using intersection theory as $\frac{1}{d!} \int_Z c_1(L)^{d}$.
    \end{itemize}
\end{remark}

\begin{remark}~\label{remark_only_fin_many_Hilbert_polynomials_standard}
    Not every polynomial $P \in \QQ[t]$ can be obtained as the Hilbert polynomial of some closed subscheme $X \subset \PP_{\Kfield}^n$ with respect to $\cO(1)$. By a theorem of Gotzmann \cite{Gotzmann1978} (see also the discussion in \cite[Section 1.8.C, Theorem 1.8.35]{Lazarsfeld_positivity_I}) there are unique integers
    \[ a_1 \geq a_2 \geq \dots \geq a_s \geq 0, \]
    such that $P(m)$ can be expressed in the form
    \[ P(m) = \binom{m+a_1}{a_1} + \binom{m+a_2-1}{a_2} + \dots + \binom{m+a_s - (s-1)}{a_s}. \]
    In particular, there are only finitely many possibilities for such $P$ of a fixed degree and bounded leading coefficient. 
\end{remark}

Using Hilbert schemes one can build a countable family of schemes parametrising all projective schemes. We need the following variant.

\begin{proposition}~\label{proposition_0_def_family_of_flat_morphisms}
    There is a countable family of smooth projective morphisms between finite type schemes over $\Kfield$
    \[ \{ \cX_i \to S_i \}_{i \in I}\]
    together with closed subschemes $\cZ_i \subset \cX_i$ such that $S_i$ are quasi-projective over $\Kfield$ and for every smooth, geometrically integral variety $X \subset \PP_{l}^n$ together with a closed subscheme $Z \subset X$ over any field $l$ extending $\Kfield$, there is an $i \in I$ and $s \in S_i(l)$ such that the inclusions $Z \subset X, (\cZ_i)_s \subset (\cX_i)_s$ are isomorphic over $l$.
\end{proposition}
\begin{proof}
    Let $H_0$ be a connected component of $\un{\Hilb}_{\Kfield}(\PP_{\Kfield}^n)$ and denote by $U \subset \PP_{\Kfield}^n \times_{\Kfield} H_0$ the intersection of the universal family of $\un{\Hilb}_{\Kfield}(\PP_{\Kfield}^n)$ with $\PP_{\Kfield}^n \times_{\Kfield} H_0$. Note that $U \to H_0$ is a flat morphism between projective varieties over $\Kfield$.

    Let $H_1 = \un{\Hilb}_{H_0}(U)$ and denote by $\cZ \subset U \times_{H_0} \un{\Hilb}_{H_0}(U)$ the universal family of $\un{\Hilb}_{H_0}(U)$. Let $S_1$ be a connected component of $H_1$ and let $\cX_1 = U \times_{H_0} S_1$. Moreover, define $\cZ_1$ to be the intersection of $\cZ$ with $\cX_{1}$.

    Let $Z \subset X \subset \PP_{l}^n$ be a given tuple of subschemes of $\PP_{l}^n$ over a field $l$ extending $\Kfield$. Let $h = [X] \in \un{\Hilb}_{\Kfield}(\PP_{\Kfield}^n)(l)$ be the corresponding point in the Hilbert scheme and assume that $h \in H_0(l)$. By the definition of the universal family $U$ we then get that $U_h \subset \PP_{l}^n$ is equal to $X$ in the sense of Remark~\ref{remark_nice_properties_hilbert_schemes}. Morphisms from $\Spec(l)$ to $H_1$ over $h \in H_0(l)$ correspond to subschemes of $U_h = X$, so there is $s \in H_1(l)$ corresponding to $Z \subset X$. Assume that $s \in S_1(l)$. Then we get a pullback diagram
    \[\begin{tikzcd}
    	Z & \cZ \\
    	X & {\cX_1} \\
    	{\Spec(\Kfield)} & {S_1}
    	\arrow[from=2-1, to=2-2]
    	\arrow["s", from=3-1, to=3-2]
    	\arrow[from=2-2, to=3-2]
    	\arrow[from=2-1, to=3-1]
    	\arrow[hook', from=1-2, to=2-2]
    	\arrow[hook', from=1-1, to=2-1]
    	\arrow[from=1-1, to=1-2]
    \end{tikzcd}\]
    In this construction we made a choice of a connected component twice. By looking at all possible such choices, and then shrinking each component obtained $S_1$ by a (relatively) open subset over which fibers of $\cX_1 \to S_1$ are smooth and geometrically integral (note that this is possible by generic smoothness, e.g. by \cite[Lemma 0C3K, Lemma 056V]{stacks-project}), we get a countable family $\{ \cX_i \to S_i \}_{i \in I}$ that we were looking for.
\end{proof}

Closely related to Hilbert schemes are Hom schemes. Intuitively, the point is that if we can form moduli of subsets of varieties, then we can also form moduli of morphisms between varieties, by looking at graphs of such morphisms. We give a functorial description below.

\begin{definition}
    Let $\cX, \cY$ be two projective schemes over $S$ with $\cX \to S$ flat. There exists a scheme over $S$ denoted by $\un{\Hom}_S(\cX, \cY)$ such that for any scheme $T \in \Sch_S$ the set of morphisms from $T$ to $\un{\Hom}_S(\cX, \cY)$ is naturally identified with
    \[ \un{\Hom}_S(\cX, \cY)(T) = \{ \textnormal{morphisms } \cX_T \to \cY_T \textnormal{ over } T \}. \]
    We call this scheme the \textit{Hom scheme} of morphisms from $\cX$ to $\cY$ over $S$. There is an open embedding $\un{\Hom}_S(\cX, \cY) \to \un{\Hilb}_S(\cX \times_S \cY)$ which on the level of functors $\Sch_S \to \Set$ is given by sending a morphism $f:\cX_T \to \cY_T$ to its graph $\Gamma_f \subset \cX_T \times_T \cY_T = (\cX \times_S \cY)_T$ (which is flat over $T$). 
\end{definition}

Again, the relative Hom scheme fits together the family of Hom schemes of morphisms between fibers of families $\cX, \cY$. More precisely, if $s \in S(\Kfield)$, then the base-change of $\un{\Hom}_S(\cX, \cY)$ via $s$ is the scheme $\un{\Hom}_\Kfield(\cX_s, \cY_s)$ and its $\Kfield$-points correspond to maps $\cX_s \to \cY_s$ over $\Kfield$.

\begin{remark}~\label{remark_closed_subscheme_of_codomain_gives_closed_hom_scheme}
    If $\cY_0 \subset \cY$ is a closed subscheme, then there exists a natural closed immersion $\un{\Hom}_S(\cX, \cY_0) \subset \un{\Hom}_S(\cX, \cY)$ of $S$-schemes. As a map of functors it sends $f:\cX_T \to (\cY_0)_T$ to the composition of $f$ with the embedding $(\cY_0)_T \subset \cY_T$.
\end{remark}

\begin{remark}~\label{remark_constant_maps_component_in_the_hom_scheme}
    Consider the morphism $\cY \to \un{\Hom}_S(\cX,\cY)$ of $S$-schemes given by the map of functors
    \[ (T \xrightarrow{g} \cY) \mapsto (\cX_T \xrightarrow{p_2} T \xrightarrow{(g, \id)} \cY_T). \]
    This map is a monomorphism and one can check that it satisfies the valuative criterion of properness. Thus it is a closed immersion. We will denote by $\un{\Hom}_S^{\con}(\cX,\cY)$ the corresponding closed subscheme of $\un{\Hom}_S(\cX,\cY)$, and by $\un{\Hom}_S^{\nc}(\cX,\cY)$ its complement.
\end{remark}

To understand the structure of the constant and non-constant parts of the Hom scheme, one can use the following.

\begin{lemma}~\label{lemma_rigidity_lemma}~\cite[11.5.12. Rigidity Lemma]{Ravi_Vakil_FOAG}
    Let $X$ be a projective, geometrically integral variety over $\KK$ with a $\KK$-point, and $Y$ be an irreducible variety over $\KK$. Fix a morphism $f:X \times_{\KK} Y \to Z$. If $q \in Y$ (a scheme theoretic point) is such that $f$ is constant on $X \times \{q\}$, then $f$ factors through $Y$.
\end{lemma}

\begin{remark}~\label{remark_connected_component_in_a_Hilbert_scheme}
    Let $X, Z$ be projective varieties over $\KK$ with $X$ geometrically integral. Consider the closed immersion $Z \simeq \un{\Hom}_{\KK}^{\con}(X,Z) \subset \un{\Hom}_{\KK}(X,Z)$. In this case, its image is actually a connected component of $\un{\Hom}_{\KK}(X,Z)$, as otherwise one could construct a map from an irreducible variety $Y \to \un{\Hom}_{\KK}(X,Z)$ whose image intersects $\un{\Hom}_{\KK}^{\con}(X,Z)$, but this would contradict the rigidity lemma above, at least after passing to the algebraic closure of $\KK$. Note that we can pass freely to the algebraic closure, as 
    \[ \un{\Hom}_{\KK}(X,Z) \otimes_{\KK} \overline{\KK} \simeq \un{\Hom}_{\overline{\KK}}(X_{\overline{\KK}},Z_{\overline{\KK}} ). \]
\end{remark}

Using Hom schemes we study the following property.
\begin{definition}
    We call a projective variety $X$ (over $\Kfield$) \textit{indecomposable} over $\KK$, if for any geometrically integral, smooth, projective varieties $X_1, X_2$ over $\KK$, any map $X_1 \times X_2 \to X$ factors through $X_1$ or $X_2$. Moreover, for a variety over a non-necessarily algebraically closed field, we call it \textit{geometrically indecomposable} if its base-change to the algebraic closure is indecomposable.
\end{definition}
\begin{lemma}~\label{lemma_indecomposability_and_discreteness_of_hom_schemes}
    Let $X_2, X$ be projective varieties over $\Kfield$ with $X_2$ geometrically integral. Consider the following properties:
    \begin{enumerate}
        \item The scheme $\un{\Hom}_{\Kfield}^{\nc}(X_2, X)$ is discrete.
        \item For every geometrically integral, projective variety $X_1$ over $\KK$ any map $X_1 \times_{\Kfield} X_2 \to X$ factors through a projection to $X_1$ or $X_2$.
    \end{enumerate}
    Then $(1) \implies (2)$ and if $\KK$ is algebraically closed, also the other implication holds.
\end{lemma}
\begin{proof}
    Note that by Remark~\ref{remark_connected_component_in_a_Hilbert_scheme}, a non-constant map $X_1 \to \un{\Hom}_{\Kfield}^{\nc}(X_2, X)$ corresponds exactly to a map $X_1 \times_{\Kfield} X_2 \to X$ over $\Kfield$ not factoring through the projections to both $X_1$ and $X_2$. 
    
    If $X_1$ is geometrically integral, then any map from $X_1$ to a discrete scheme over $\KK$ factors through a morphism from $\Spec(\KK)$. Thus discreteness of $\un{\Hom}_{\Kfield}^{\nc}(X_2, X)$ implies that any map $X_1 \to \un{\Hom}_{\Kfield}^{\nc}(X_2, X)$ factors through a $\KK$-point of $\un{\Hom}_{\Kfield}^{\nc}(X_2, X)$ which corresponds to the fact that any map $X_1 \times_{\Kfield} X_2 \to X$ factors through a projection to $X_2$.

    On the other hand, in the case $\KK$ is algebraically closed, if $\un{\Hom}_{\Kfield}^{\nc}(X_2, X)$ is not discrete, then there is an integral variety $X_1$ over $\KK$ with a non-constant map $X_1 \to \un{\Hom}_{\Kfield}^{\nc}(X_2, X)$ which contradicts $(2)$.
\end{proof}

\begin{proposition}~\label{proposition_geometric_indecomposability}
    Let $X$ be a projective variety over $\Kfield$. If $X$ is geometrically indecomposable, then $X_l$ is indecomposable over $l$ for any field extension $\Kfield \subset l$. Moreover, $X$ is geometrically indecomposable if and only if $X_l$ is.
\end{proposition}
\begin{proof}
    Let $\{\cX_i \to S_i\}_{i \in I}$ be the set from Proposition~\ref{proposition_0_def_family_of_flat_morphisms}. By Lemma~\ref{lemma_indecomposability_and_discreteness_of_hom_schemes} (and resolution of singularities),  indecomposability of $X$ is equivalent to the following condition: for every morphism $\cX_i \to S_i$ from the above collection, if $\cX \to S$ is its base-change to $\Kfield$, then the morphism $H:=\un{\Hom}_S^{\nc}(\cX, X_S) \to S$ has discrete fibers over $\overline{\Kfield}$-points. Note that by choosing relatively ample line bundle on $\cX$ and an ample line bundle on $X$, the scheme $H \to S$ can be presented as a disjoint sum of quasi-projective components over $S$. Hence, $H \to S$ is locally quasi-finite, and one can use \cite[Lemma 06RT]{stacks-project} to get that for any field extension $\Kfield \subset l$, the fibers over $l$-points of $S$ are also discrete, which by Lemma~\ref{lemma_indecomposability_and_discreteness_of_hom_schemes} proves that $X_l$ is indecomposable over $l$.
    
    To see the moreover part of the proposition, we only need to see that if the fibers over $\overline{l}$-points of $S$ of morphisms $H \to S$ are discrete, then also fibers over $\overline{k}$-points of $S$ are discrete. This is clear as we can embed $\overline{k} \subset \overline{l}$.
\end{proof}

\section{Rigidity in algebraic geometry}~\label{sec:rigidity}

In this section we present some results from algebraic geometry that one can use to find examples of varieties $V$ for which the theory of fields without points in $V$ has a model companion and some further nice properties. In the case of \cite{johnson2023curveexcluding} the variety $V=C$ is a genus $\geq 2$ curve and this assumption is used to construct some uniformly definable sets (moduli) of rational functions from smooth varieties to $C$. Recall that the genus of a smooth curve $C$ is the dimension of the space of global differential forms on $C$. It this section we consider some other positivity assumptions on the vector bundle $\Omega_V$ of differential forms on $V$. We outline how such assumptions imply existence of (moduli) spaces of rational functions from smooth varieties to $V$. We mostly follow \cite[Chapter 6]{Lazarsfeld_positivity_II} and \cite{algebraic_hyperbolicity} in our presentation.

Let $X$ be a smooth projective variety over an algebraically closed field $\KK$. If $X$ is a curve, then the cotangent bundle is in fact a line bundle on $X$ and it makes sense to ask whether it is ample. This is not the case in higher dimensions. Still, Hartshorne defined a notion of ampleness for vector bundles of arbitrary rank. For a vector bundle $E$ on $X$, we denote by $\PP(E)$ the projectivisation of $E$, i.e., the set of codimension one hyperplanes passing through zero in fibers of $E$ over $X$. 
\begin{definition}
    A vector bundle $E$ on $X$ is ample, if the line bundle $\cO_{\PP(E)}(1)$ on $\PP(E)$ is ample. Note that $\cO_{\PP(E)}(1)$ is always relatively ample over $X$.
\end{definition}
For a smooth curve $C$ its cotangent bundle is ample if and only if the genus of $g$ is bigger than one. In higher dimensions ampleness of the cotangent bundle intuitively means that the variety is very rigid. To make this statement more precise we introduce some rigidity notions in algebraic geometry.
\begin{definition}~\label{definition_hyperbolicity_notions_for_a_variety}
    Let $X$ be a smooth projective variety over $\KK$. Fix an ample line bundle $L$ on $X$. Below we only consider maps over $\KK$. We say that $X$ is:
    \begin{enumerate}
        \item \textit{Brody hyperbolic} if $\KK=\CC$ and there is no non-constant holomorphic map $f:\CC \to X$, where here we treat $X$ as a complex manifold;
        \item \textit{algebraically hyperbolic} if there exists $a,b \in \RR_{\geq 0}$ such that for every smooth curve $C$ over $\KK$ and a finite map $f:C \to X$ we have
        \[ \int_C c_1(f^*L) \leq a\cdot g(C) + b, \]
        where $g(C)$ is the genus of $C$;
        \item \textit{bounded} if, for every normal integral projective scheme $Y$ over $\KK$, the scheme $\un{\Hom}_{\KK}(Y,X)$ is of finite type over $\KK$ (in particular it is definable in $\ACF_0$ over $\KK$);
        \item \textit{$1$-bounded} if, for every smooth projective curve $C$ over $\KK$, the scheme $\un{\Hom}_{\KK}(C,X)$ is of finite type over $\KK$;
        \item \textit{groupless} if, for any connected finite type group scheme $G$ over $\KK$, there are no non-constant maps $G \to X$;
        \item \textit{pure} (over $\KK$) if, for any smooth variety $T$ over $\KK$, any rational map $T \dashrightarrow X$ extends to a regular map $T \to X$.
    \end{enumerate}
\end{definition}
\begin{remark}~\label{remark:weak_Lang_Vojta}
    In fact the following implications hold:
    \[ (1) \implies (2) \implies (3) \iff (4) \implies (5) \implies (6), \]
    where $(1) \implies (2)$ only makes sense if $\KK=\CC$. For the proofs of many of these implications, see \cite{algebraic_hyperbolicity}. Also, conditions $(1) - (5)$ are conjectured to be equivalent, see \cite[Section 9]{algebraic_hyperbolicity}. This is motivated by the case of curves, where $(1)-(5)$ are equivalent to the genus of the curve being bigger or equal two. Note that an elliptic curve over $\KK$ is pure, but not groupless.
\end{remark}
\begin{remark}~\label{remark_many_hyperbolicity_assumptions_are_geometric}
    Properties $(2)-(6)$ are all geometric, in the sense that for any extension $\KK \subset l$ of algebraically closed fields, they hold for $X$ over $\KK$ if and only if they hold for $X_l$ over $l$. For the proofs see \cite[Theorem 1.5, Theorem 1.11]{algebraic_hyperbolicity} and the next remark.
\end{remark}
\begin{remark}~\label{remark_equivalent_conditions_on_boundedness_etc}
    It is worth to note that some of the above properties can be described differently (assuming notation from Definition~\ref{definition_hyperbolicity_notions_for_a_variety}), namely:
    \begin{itemize}
        \item $X$ is pure if and only if every morphism $\PP^1 \to X$ is constant;
        \item $X$ is groupless if and only if for all abelian varieties $A$ over $\KK$, all morphisms $A \to X$ are constant;
        \item $X$ is bounded (or equivalently $1$-bounded) if and only if there exist real numbers $\alpha(g)$ depending on $X, L$, such that for all smooth projective curves $C$ over $\KK$ and maps $f:C \to X$ we have
        \[ \int_C c_1(f^*L) \leq \alpha(g(C)), \]
        where $g(C)$ is the genus of $C$.
    \end{itemize}
    For the proofs, see \cite{algebraic_hyperbolicity}. Also, let us remark that in \cite[Definition 3.1]{algebraic_hyperbolicity} there is an additional condition on the rational map $T \dashrightarrow X$, namely that it is defined on an open subset $U \subset T$ with complement of codimension $\geq 2$. This assumption is redundant for projective $X$, for example by the curve-to-projective theorem from \cite[Exercise 17.5.B]{Ravi_Vakil_FOAG}.
\end{remark}

If the base-field is not necessarily algebraically closed, we use the following variant of pureness.

\begin{definition}
    Let $\KK_0$ be any characteristic zero field (not necessarily algebraically closed). A projective variety $\widetilde{V}$ over $\KK_0$ is \textit{geometrically pure}, if the base-change $\widetilde{V}_{\overline{\KK_0}}$ is pure. 
\end{definition}

Geometric pureness is preserved under base-change in the following strong sense.

\begin{proposition}~\cite{purity_even_without_alg_closedness}~\label{proposition_geometrical_pureness_under_base_change}
    Let $\widetilde{V}$ be a geometrically pure, projective variety over a field $\KK_0$. Let $\KK_0 \subset l$ be any field extension. Then $\widetilde{V}_l$ satisfies the definition of pureness over $l$, i.e., for any smooth variety $T$ over $l$, any rational map $T \dashrightarrow \widetilde{V}_l$ extends to a regular one.
\end{proposition}
\begin{proof}
    This is a direct consequence of \cite[Proposition 6.2]{purity_even_without_alg_closedness} if we put $X = \widetilde{V}_l \times_l T \to T = S$ in the notation there.
\end{proof}

There exists a relative version of $1$-boundedness which we present below.
\begin{definition}
    Let $\Delta \subset X$ be a closed subscheme. We say that $X$ is \textit{$1$-bounded modulo $\Delta$}, if for every smooth projective curve $C$ over $\KK$, the scheme $\un{\Hom}_{\KK}(C,X) \setminus \un{\Hom}_{\KK}(C,\Delta)$ is of finite type over $\KK$. Similarly, we say that $X$ is \textit{bounded modulo $\Delta$}, if for every normal projective variety $Y$ over $\KK$, the scheme $\un{\Hom}_{\KK}(Y,X) \setminus \un{\Hom}_{\KK}(Y,\Delta)$ is of finite type over $\KK$. 
    
    In the case where $\KK$ is not algebraically closed, we say that $X$ is \textit{geometrically bounded modulo $\Delta$} if after base-change to the algebraic closure of $\KK$ the corresponding pair satisfies the definition of being bounded modulo a subscheme.
\end{definition}
In the non-relative setting these two notions are the same. In the relative setting this is only known under an additional assumption. 
\begin{lemma}~\cite[Lemma 9.9]{vanbommel2019boundedness}~\label{lemma_1_bdd_vs_bdd}
    Assume that $\KK$ is uncountable. Then $X$ being $1$-bounded modulo $\Delta$ implies that $X$ is bounded modulo $\Delta$.
\end{lemma}

Moreover, similarly as in the non-relative case, it is true that boundedness modulo $\Delta$ implies existence of a uniform bound on degrees of curves mapped into $X$ with the image not in $\Delta$.
\begin{theorem}~\cite[Theorem 1.26]{vanbommel2019boundedness}~\label{theorem_boundedness_gives_uniform_bound_for_curves}
    Let $X$ be bounded modulo $\Delta$ and fix an ample line bundle $L$ on $X$. Then there exist real numbers $\alpha(g)$ depending on $X, \Delta, L$, such that for all smooth projective curves $C$ over $\KK$ of genus $g$ and for all maps $f:C \to X$ with $f(C) \not\subset \Delta$ we have
    \[ \int_C c_1(f^*L) \leq \alpha(g(C)). \]
\end{theorem}

Now we present a result that connects ampleness of the cotangent bundle to the rigidity conditions introduced above.

\begin{theorem}{(Kobayashi's theorem~\cite{Kobayashi_theorem}, \cite[Theorem 6.3.26]{Lazarsfeld_positivity_II})}
    Assume that $X$ is a smooth variety with an ample cotangent bundle $\Omega_X$. Then $X$ is algebraically hyperbolic. Moreover, if $\KK=\CC$, then $X$ is Brody hyperbolic.
\end{theorem}

The other implication does not hold and one can conclude it for example from the following result originally proved by Kalka, Shiffman and Wong in \cite[Corollary 4]{101307mmj1029002559} in the analytic setting.

\begin{theorem}~\cite[Corollary 6.3.30]{Lazarsfeld_positivity_II}~\label{theorem_ample_cotangent_implies_indecomposability}
    Assume that $X$ is a smooth variety with an ample cotangent bundle $\Omega_X$. Let $Y$ be an integral projective variety. Then the scheme $\un{\Hom}_{\KK}^{\nc}(Y,X)$ is finite.
\end{theorem}

So in particular, if $C$ is a curve of genus greater than one, then $C \times C$ is algebraically hyperbolic (even Brody hyperbolic if $\KK=\CC$), but does not have ample cotangent bundle. Indeed, note that for each $a \in C(\KK)$ there is a morphism $C \to C \times C$ given by $b \mapsto (b, a)$, hence the scheme $\un{\Hom}_{\KK}^{\nc}(C,C \times C)$ is infinite.

Note that from Theorem~\ref{theorem_ample_cotangent_implies_indecomposability} it follows that smooth varieties with ample cotangent line bundles are indecomposable. We also observe the following.

\begin{lemma}~\label{lemma_indecomposable_is_groupless}
    Assume that $X$ is indecomposable. Then $X$ is groupless.
\end{lemma}
\begin{proof}
    Assume for a contradiction that there is a non-constant map $A \to X$ from an abelian variety $A$ over $\KK$. If we compose this map with addition $A \times_{\KK} A \to A$ we get a map
    \[ A \times_{\KK} A \to X \]
    which does not factor through any projection to $A$.
\end{proof}

\begin{remark}
    By the above lemma, an indecomposable variety is pure. Thus, in the definition of indecomposability we could equivalently impose, that all rational maps $X_1 \times X_2 \dashrightarrow X$ factor through $X_1$ or $X_2$.
\end{remark}

Motivated by the conjectural equivalence of conditions $(1)-(5)$ from Definition~\ref{definition_hyperbolicity_notions_for_a_variety} we ask the following.
\begin{question}
    Assume that $X$ is indecomposable. Is $X$ bounded?
\end{question}

Note that this question has a positive answer in dimension one, where indecomposability and all other hyperbolicity properties are equivalent to having genus at least two. For some partial results in arbitrary dimension see \cite{vanbommel2019boundedness}.

A plenty of examples of smooth varieties with ample cotangent bundle is provided by the following theorem proved by Brotbek and Darondeau.
\begin{theorem}~\cite[Theorem 0.1]{complete_int_have_ample_cotangent}~\label{theorem_complete_int_have_ample_cotangent}
    Let $M$ be an $N$-dimensional smooth projective variety over $\KK$ with a very ample line bundle $\cO_M(1)$. Pick an integer $N/2 \leq c \leq N$ and $\delta \in (\ZZ_{>0})^c$. Then there exists $\nu(\delta) \in \QQ$ such that for all multi-degrees $(d_1,\dots, d_c) = \nu \cdot (\delta_1,\dots,\delta_c)$ with $\nu \in \QQ$ such that $\nu \geq \nu(\delta)$, the complete intersection of general hypersurfaces $H_1 \in |\cO_M(d_1)|,\dots, H_c \in |\cO_M(d_c)|$ has ample cotangent bundle.
\end{theorem}
For some other examples of varieties with ample cotangent bundle, see \cite[Section 6.3]{Lazarsfeld_positivity_II}. Let us also note that strong positivity assumptions on $\Omega_X$ can also yield finiteness results like the following.
\begin{theorem}~\cite{Moriwaki1994RemarksOR}
    Let $X$ be a smooth variety over a number field $\KK_0$ such that $\Omega_X$ is ample and globally generated (i.e., the canonical map $H^0(X, \Omega_X) \otimes_{\KK_0} \cO_X \to \Omega_X$ is surjective). Then $X(\KK_0)$ is finite.
\end{theorem}
Moreover, Lang's conjecture predicts that in the above theorem global generatedness is redundant, see \cite[Conjecture D]{Moriwaki1994RemarksOR}. The above result ultimately relies on the following theorem of Faltings.
\begin{theorem}~\cite{Faltings}~\label{theorem_faltings}
    Let $X$ be a closed subvariety of an abelian variety $A$ over $\KK$. Then there is a closed proper subscheme $Y \subset X$ such that $X$ is \textit{Mordellic modulo} $Y$, i.e., for any finitely generated field $\KK_0$ and any model $X_0$ for $X$ over $\KK_0$, the set $X_0(\KK_0) \setminus Y$ is finite.
\end{theorem}
In particular, using this theorem we can get plenty of examples of varieties that we will later see to be excludable.
\begin{corollary}~\label{corollary_random_subvarieties_of_abelian_varieties_provide_examples}
    Let $A$ be an abelian variety over $\KK$ and let $X$ by obtained by a sufficiently general intersection of ample line bundles as in Theorem~\ref{theorem_complete_int_have_ample_cotangent}. Then there is a finitely generated subfield $\KK_0 \subset \KK$, a $\KK_0$-model $X_0$ of $X$ and a closed proper subscheme $Y_0 \subset X_0$, such that $(X_0 \setminus Y_0)(\KK_0)$ is empty, and the cotangent bundle of $X_0$ is ample. Moreover, we can assume that $X_0$ is geometrically integral and smooth.
\end{corollary}

\section{Model companion via Hilbert schemes}~\label{sec:model_comp}

In this section, we fix a characteristic zero field $\Kfield_0$ (not necessarily algebraically closed) and use the notation $\Kfield$ for an extension of the field $\Kfield_0$. Fix a geometrically integral projective variety $\widetilde{V}$ over $\Kfield_0$. Let $V_0 \subset \widetilde{V}$ be a closed proper subscheme such that $\widetilde{V}$ geometrically is bounded modulo $V_0$ and denote $V = \widetilde{V} \setminus V_0$. Moreover, fix a smooth projective morphism between varieties $\cX \to S$ over $\Kfield_0$ with geometrically integral fibers and with a closed subscheme $\cZ \subset \cX$. The below proposition is a variant of \cite[Theorem 6.3.29]{Lazarsfeld_positivity_II} in families.

\begin{proposition}~\label{proposition_quasi_projectivity_of_H}
    There is a quasi-projective morphism $H \to S$, such that for each $s \in S(\Kfield)$ the pullback $H_s$ is isomorphic to $\un{\Hom}_{\Kfield}^{\nc}(\cX_s, \widetilde{V}_{\Kfield}) \setminus \un{\Hom}_{\Kfield}(\cX_s,(V_0)_{\Kfield})$.
\end{proposition}
\begin{proof}
    By Remark~\ref{remark_closed_subscheme_of_codomain_gives_closed_hom_scheme} and Remark~\ref{remark_constant_maps_component_in_the_hom_scheme} the scheme $\un{\Hom}_S(\cX,\widetilde{V} \times_{\Kfield_0} S)$ has two closed subschemes $\un{\Hom}_S^{\con}(\cX,\widetilde{V} \times_{\Kfield_0} S)$ and $\un{\Hom}_S(\cX,V_0 \times_{\Kfield_0} S)$. Take $H$ to be their complement, i.e., $H = \un{\Hom}_S^{\nc}(\cX,\widetilde{V} \times_{\Kfield_0} S) \setminus \un{\Hom}_S(\cX,V_0 \times_{\Kfield_0} S)$. Note that fibers of $H$ over points of $S$ are as claimed, so it is enough to check quasi-projectivity over $S$. 

    Note that $H$ is an open subscheme of $\un{\Hom}_S(\cX,\widetilde{V} \times_{\Kfield_0} S)$ which furthermore is naturally an open subscheme of the Hilbert scheme \[\un{\Hilb}_S(\cX \times_S (\widetilde{V} \times_{\Kfield_0} S)) = \un{\Hilb}_S(\cX \times_{\Kfield_0} \widetilde{V}). \]
    If we fix a relatively ample line bundle on $\cX$ over $S$ and an ample line bundle on $\widetilde{V}$, this scheme can be written as a disjoint sum
    \[ \un{\Hilb}_S(\cX \times_{\Kfield_0} \widetilde{V}) = \bigcup_{P} \un{\Hilb}_S^P(\cX \times_{\Kfield_0} \widetilde{V}), \]
    where $P$ varies among Hilbert polynomials with respect to the relatively ample line bundle (over $S$) induced on $\cX \times_{\Kfield_0} \widetilde{V}$. Moreover, all $\un{\Hilb}_S^P(\cX \times_{\Kfield_0} \widetilde{V})$ are projective over $S$, and by Remark~\ref{remark_only_fin_many_Hilbert_polynomials_standard} there are only finitely many Hilbert polynomials with fixed degree and bounded leading coefficient. Since all components $\un{\Hilb}_S^P(\cX \times_{\Kfield_0} \widetilde{V})$ are projective over $\Kfield_0$, to prove that $H$ is quasi-projective, it is enough to show that for $\Kfield = \overline{\Kfield_0}$, if
    \[ [f]=(f:\cX_{s} \to \widetilde{V}_{\Kfield}) \in H(\Kfield), \]
    then there is a bound on the leading coefficient of $P$ such that $[f] \in \un{\Hilb}_\Kfield^P(\cX_{s} \times_{\Kfield} \widetilde{V}_\Kfield)$, and this bound is independent of $f$ and of the image $s \in S(\Kfield)$ of $[f]$. Note that by Remark~\ref{remark_nice_properties_hilbert_schemes} (the last point) the degree of $P$ such that $[f] \in \un{\Hilb}_\Kfield^P(\cX_{s} \times_{\Kfield} \widetilde{V}_\Kfield)$ is fixed and equal to $\dim \cX_s$, because the dimension of the graph of $f$ is equal to the dimension of $\cX_s$. 

    Fix $[f]$ and $s \in S(\Kfield)$ as above. Let $\cL$ be a relatively ample line bundle on $\cX$ over $S$ and let $L$ be an ample line bundle on $\widetilde{V}$. Note that the degree of the Hilbert polynomial of $[f] \in \un{\Hilb}_\Kfield^P(\cX_{s} \times_{\Kfield} \widetilde{V}_\Kfield)$ can be calculated as the following intersection number
    \[ \frac{1}{d!} \int_{\cX_{s}} c_1(\cL_{s} \otimes f^*(L_{\Kfield}))^d, \]
    where $d$ is the relative dimension of $\cX$ over $S$. By using the same reasoning as in \cite[Proof of Theorem 6.3.29]{Lazarsfeld_positivity_II}, if we can bound the number
    \[ \int_{\cX_{s}} c_1(\cL_{s})^{d-1} c_1(f^*(L_{\Kfield})) \]
    independently of $s$ and $[f]$, then we will be done (one can also argue the sufficiency of such bound by the Khovanskii-Teissier inequalities, see Remark~\ref{Khovanskii_Teissier_inequalities}). 
    
    Now we use Bertini theorem to take sections $t_i$ in $H^0(\cX_s, \cL_s)$ for $i=1, \dots, d-1$ such that the zero sets of $t_i$'s form a complete intersection smooth curve $C \subset \cX_s$ which is not mapped to $(V_0)_{\Kfield}$ by $f$ (this curve depends on $s$ and $f$). Let $g$ be the genus of $C$. 
    By Theorem~\ref{theorem_boundedness_gives_uniform_bound_for_curves} we get that there is a number $\alpha(g)$, such that
    \[ \int_{\cX_{s}} c_1(\cL_{s})^{d-1} c_1(f^*(L_{\Kfield})) = \int_{C} c_1(f^*(L_{\Kfield})) \leq \alpha(g). \]
    Let $e = \int_{\cX_s} c_1(\cL_{s})^{d}$ and note that it is independent of $s$ by flatness of the family $\cX$ over $S$. We can bound $g$ using \cite{GLP-bound} in the following way
    \[ g \leq e^2 - 2e + 1.  \]
    This yields the bound
    \[ \int_{\cX_{s}} c_1(\cL_{s})^{d-1} c_1(f^*(L_{\Kfield})) \leq \alpha(g) \leq \max_{g \leq e^2 - 2e + 1} \alpha(g), \]
    which is independent of $[f]$ and $s$. This finishes the proof.
\end{proof}

\begin{remark}
    Note that the above proof actually shows that if the conclusion of Theorem~\ref{theorem_boundedness_gives_uniform_bound_for_curves} is satisfied for the base-change to the algebraic closure of the pair $V_0 \subset \widetilde{V}$, then $\widetilde{V}$ is geometrically bounded modulo $V_0$. Moreover, if in the above proof we assume that $V_0$ is empty and the fiber $\cX_s$ is fixed, then the curve $C$ does not depend on $f$. Thus in that case the above proof recovers the implication ``$1$-bounded implies bounded" from \cite[Theorem 9.3]{algebraic_hyperbolicity}.
\end{remark}

\begin{lemma}~\label{lemma_0_definability_for_one_family}
    Assume that $\widetilde{V}$ is geometrically pure. Let $H \to S$ be the quasi-projective morphism constructed in the above proposition. The following conditions on the field $\Kfield$ are equivalent.
    \begin{enumerate}
        \item For any pair $Z \subset X/\Kfield$ appearing as a $\Kfield$ fiber of $\cZ \subset \cX \to S$, either there is a non-constant rational map $f : X \dashrightarrow \widetilde{V}_{\Kfield}$ over $\Kfield$ not factoring through the inclusion $(V_0)_{\Kfield} \subset \widetilde{V}_{\Kfield}$ or $(X \setminus Z)(\Kfield)$ is non-empty.
        \item The following sentence is true:
        \[(\forall s \in S(\Kfield))(H_s(\Kfield)\neq \varnothing \vee (\cX \setminus \cZ)_s(\Kfield) \neq \varnothing ).\]
    \end{enumerate}
\end{lemma}
\begin{proof}
    First, note that, by Proposition~\ref{proposition_geometrical_pureness_under_base_change}, it follows that $\widetilde{V}_{\Kfield}$ itself satisfies the extension principle from the definition of pureness. In particular, in $(1)$ above we can only quantify over regular maps $f:X \to \widetilde{V}_{\Kfield}$ over $\Kfield$. This finishes the proof of this lemma, by the characterisation of fibers of the scheme $H$ from Proposition~\ref{proposition_quasi_projectivity_of_H}.
\end{proof}

\begin{corollary}
Let $\widetilde{V_0}$ be a geometrically integral projective variety over $\Kfield_0$ and $V_0 \subset \widetilde{V}$ be a closed proper subscheme such that $\widetilde{V}$ geometrically is bounded modulo $V_0$. Let $V = \widetilde{V} \setminus V_0$.    Consider the property $(\mathrm{II}_{V})$ for a field $\Kfield$, i.e., 
    \item[$(\mathrm{II}_{V})$] For any geometrically integral
    variety $X/\Kfield$, either there is a non-constant rational map $f : X
    \dashrightarrow V$ over $\Kfield$ or $X(\Kfield)$ is Zariski dense in $X$.\\
    This is an elementary property of the field $\Kfield$. More precisely, it is elementary in $L_{\mathrm{rings}}(\KK_0)$, i.e., in the language of rings expanded with constant symbols for elements in $\KK_0$.
\end{corollary}
\begin{proof}
    This follows from Proposition~\ref{proposition_0_def_family_of_flat_morphisms} and Lemma~\ref{lemma_0_definability_for_one_family}.
\end{proof}
Over all this proves the following statement.
\begin{theorem}~\label{theorem_main_result_of_the_paper}
    Let $\widetilde{V}$ be a geometrically integral, geometrically pure, projective variety over $\Kfield_0$. Let $V_0 \subset \widetilde{V}$ be a closed subscheme such that $\widetilde{V}$ is geometrically bounded modulo $V_0$. Let $V = \widetilde{V} \setminus V_0$. Then $V$ is excludable.
\end{theorem}

\begin{remark}
    In the above corollary one can skip the assumption that $\widetilde{V}$ is geometrically pure. This follows from the proof of Proposition~\ref{proposition_quasi_projectivity_of_H} and the proof of existence of model companion for curve excluding fields in \cite{johnson2023curveexcluding} (using Chow schemes rather than Hilbert schemes). We leave this as an exercise to the reader.
\end{remark}

\begin{corollary}~\label{corollary_random_subvarieties_are_excludable}
    Let $V_0 \subset \widetilde{V}$ be a pair obtained as in Corollary~\ref{corollary_random_subvarieties_of_abelian_varieties_provide_examples} (with $V_0=Y_0, \widetilde{V}=X_0$). Then $V=\widetilde{V} \setminus V_0$ is excludable.
\end{corollary}
\begin{remark}\label{rmk:computability}
In general, if $\widetilde{V}$ is only assumed to be 1-bounded/bounded, it is unclear if the axiomatisation of $V\XF$ is computable. The computability of the set of the axioms of $V\XF$ relies on the computability of the function $\alpha$ from the data of $X,L$ in Remark~\ref{remark_equivalent_conditions_on_boundedness_etc}. If $\widetilde{V}$ is algebraically hyperbolic, then the axiomatisation above is indeed computable.

In some sense, one can use the complexity/computability of axiomatisation in $V\XF$ as a measurement of the hyperbolicity of $V$. However, there might not be any real examples fitting in between as suggested by the conjecture that groupless implies algebraically hyperbolic.
\end{remark}

\section{$V\XF$ for indecomposable $V$}~\label{sec:indecomp}

Throughout this section, we fix a characteristic zero base-field $\Kfield_0$ and let $V$ be a quasi-projective excludable variety over $\Kfield_0$. Moreover (if not stated otherwise) assume that $V$ has a projective compactification $\widetilde{V}$ which is geometrically integral and geometrically indecomposable. We study the properties of $V\XF$ for such $V$.  Note that this is the case when $\widetilde{V}$ is $X_0$ from Corollary~\ref{corollary_random_subvarieties_of_abelian_varieties_provide_examples}. We freely use Proposition~\ref{proposition_geometric_indecomposability} and Proposition~\ref{proposition_geometrical_pureness_under_base_change} for base-changes of $\widetilde{V}$. Slightly abusing notation, let $V\XF_\forall$ be the theory $T_V$ of
$\KK_0$-fields that avoid $V$, which explicitly means that $V(K) = \varnothing$. We skip the proofs that follow exactly as the curve case in \cite{johnson2023curveexcluding}.

\begin{proposition}
    Let $V$ be an arbitrary excludable quasi-projective variety with a smooth projective completion $\widetilde{V}$.  If $\KK \models V\XF$, then $\KK$ is large if and only if $\widetilde{V}(\KK) = \varnothing$.
\end{proposition}

\begin{lemma}~\label{lemma_amalgamation_of_models_of_VXF_forall}
    If $\KK$ is a field and $l_1,l_2$ are two regular extensions satisfying $V\XF_\forall$, then $\Frac(l_1 \otimes_{\KK} l_2) \models V\XF_\forall$.
\end{lemma}
\begin{proof}  
This is essentially a restatement of the geometrically indecomposability of $\widetilde{V}$.
    We may assume that $l_1$ and $l_2$ are finitely generated over $\KK$, hence $l_i=\KK(X_i)$ for some geometrically integral, smooth (by resolution of singularities) $X_i/\KK$. Since $l_i$ avoids $V$, there is no non-constant rational map $X_i\dashrightarrow V_\KK$ over $\KK$. Assume for a contradiction, that there is a rational map $f:X_1 \times_\KK X_2\dashrightarrow V_\KK$ over $\KK$. Since $\widetilde{V}$ is geometrically pure, there is a regular extension $f:X_1 \times_\KK X_2 \to \widetilde{V}_\KK$ over $\KK$. As $\widetilde{V}$ is geometrically indecomposable, $f$ factors through the projection to $X_1$ or $X_2$. This gives a contradiction. Hence $\mathrm{Frac}(l_1\otimes_\KK l_2)=\KK(X_1\times X_2)$ avoids $V$.
\end{proof}

The above lemma is the replacement for~\cite[Lemma 4.3, 4.4]{johnson2023curveexcluding}, which enables us to generalise (without any changes) some results from the curve case in \cite{johnson2023curveexcluding} to $V\XF$. For example, we immediately get the following. 
\begin{lemma}\label{lemma_partial_elem}
  Let $k$ be a $k_0$-field and $l_1, l_2$ be two regular extensions of
  $k$ satisfying $V\XF$.  Then the identity map from $k$ to $k$ is
  partial elementary between the $L_i$'s.  In other words,
  $l_1\equiv_k l_2$.
\end{lemma}
For the remainder of the statements, we list the generalisations here and point the readers to the proofs in~\cite{johnson2023curveexcluding}. Let us recall the terminologies involved before stating the theorem.

For a field extension $\KK_0 \subset l$, recall that $\Abs_{\KK_0}(l)$ for the field $l \cap \overline{\KK_0}$. Recall that an expansion of field $(K,+,\cdot,...)$ is \textit{algebraically bounded over (a subfield) $F$}  if for any formula $\varphi(\bar{x},y)$, there are finitely many polynomials
		$P_1,\ldots,P_m \in F[\bar{x},y]$ such that for any $\bar{a}$, \emph{if}
		$\varphi(\bar{a},K)$ is finite, then $\varphi(\bar{a},K)$ is contained in
		the zero set of $P_i(\bar{a},y)$ for some $i$ such that $P_i(\bar{a},y)$ does not vanish.  We say that $K$
		is \emph{algebraically bounded} if it is algebraically bounded over
		$K$~\cite{lou-dimension}. By~\cite{JohYe_geom}, algebraically bounded over the prime field is the same as very slim in the sense of~\cite{JK-slim}. We say a field $k$ is \emph{Hilbertian} if there is $k\preceq k'$ and $t\in k'\setminus k$ such that $k(t)$ is relatively algebraically closed in $K'$. Lastly, recall that for a field $k$, its absolute Galois group $\Gal(k)$ is \emph{$\omega$-free}\footnote{This definition is slightly different from the usual definition (for example in~\cite{field-arithmetic}). Our definition is the first-order characterization of having $\omega$-free Galois group in the classical sense.}, if for any
  surjective homomorphism of finite groups $\alpha : H \to G$ and any
  continuous surjective homomorphism $\beta : \Gal(k) \to G$, there is
  a continuous surjective homomorphism $\gamma : \Gal(k) \to H$ making
  the diagram commute:
  \begin{equation*}
    \xymatrix{ \Gal(k) \ar@{-->}[r]^-{\gamma} \ar[dr]_-{\beta} & H \ar[d]^{\alpha} \\ & G.}
  \end{equation*}

\begin{proposition}~\label{theorem_properties_of_VXF}
    The theory $V\XF$ has the following properties.
    \begin{enumerate}
        \item Let $l_1, l_2 \models V\XF$. Then $l_1 \equiv l_2$ if and only if $\Abs_{\KK_0}(l_1) \simeq \Abs_{\KK_0}(l_2)$ as $\KK_0$-fields. Moreover, for any algebraic extension $\KK_0 \subset \KK$ avoiding $V$, there is a model $l \models V\XF$ such that $\Abs_{\KK_0}(l) = \KK$.
        \item $V\XF$ has quantifier elimination if we expand by the predicates
        \begin{equation*}
            \mathrm{Sol}_n(x_0,...,x_n)\leftrightarrow \exists y \, x_0+yx_1+...+y^nx_n=0.    
        \end{equation*}  
        \item Model-theoretic and field-theoretic (relative) algebraic closure coincide in models of $V\XF$. In particular, models of $V\XF$ are algebraically bounded over $k_0$.
        \item $V\XF$ is a geometric theory, i.e., $\acl$ satisfies the exchange property and $\exists^\infty$ is uniformly eliminated.
         \item Every proper finite extension of models of $V\XF$ is a Hilbertian PAC field.
         \item Models of $V\XF$ are Hilbertian. Moreover, the absolute Galois group is $\omega$-free.
         \item In the axiomatization of $V\XF$ one can in fact skip the axiom $(\mathrm{I}_V)$. In other words, fields $\KK$ that avoid $V$ and satisfy $(\mathrm{II}_V)$ are models of $V\XF$.
         \item Model-theoretically, $V\XF$ is NSOP$_4$ and TP$_2$.
    \end{enumerate}
\end{proposition}
\begin{proof}

Essentially all of the proofs here follows from the proof of the analogue statement in~\cite{johnson2023curveexcluding} replacing \cite[Lemma 4.5]{johnson2023curveexcluding} with Lemma~\ref{lemma_partial_elem}.

The proof of (1) follows exactly the same as~\cite[Theorem 4.6, 4.16]{johnson2023curveexcluding}, replacing the usage of Lemma 4.5 with~Lemma~\ref{lemma_partial_elem}. 

For (2), it follows word for word as in the proof of~\cite[Corollary 4.7]{johnson2023curveexcluding}.

For (3), it is~\cite[Theorem 4.9]{johnson2023curveexcluding}.

(4) is the same as~\cite[Corollary 4.12]{johnson2023curveexcluding}.

(5) and (6) are~\cite[Theorem 4.23, 4.27]{johnson2023curveexcluding} and~\cite[Section 8]{johnson2023curveexcluding}. 

(7) follows from (5).

(8) is the same as~\cite[Corollary 4.28] {johnson2023curveexcluding} and~\cite[Theorem 6.7]{johnson2023curveexcluding} relying on the fact that both $V\XF$ and $\omega\mathrm{PAC}_0$ admit quantifier elimination in the same language as in (2).
\end{proof}
Lastly, we have the following description of the decidability of the incomplete theory $V\XF$. The proof is the same as in~\cite[Theorem 4.20]{johnson2023curveexcluding} combined with~Remark~\ref{rmk:computability}.
\begin{proposition}
Let $k_0$ be a number field and $\widetilde{V}$ be algebraically hyperbolic over $k_0$ and $V$ as before. Then the following two decision problems are Turing equivalent:
  \begin{enumerate}
      \item Given a sentence $\phi$ in the language of $k_0$-algebras, determine whether $V\XF \vdash \phi$.
      \item Given a finite extension $k/k_0$, determine whether $V(k) = \varnothing$.
  \end{enumerate}
\end{proposition}

\section{Remarks and Questions}~\label{sec:rem_and_quest}

We start by explaining some reasons why in general we cannot hope for the existence the model companion of fields excluding $V$ in positive characteristic. 

Let $\KK_0$ be a perfect characteristic $p$ field properly containing the algebraic closure of a finite field. Let $V$ be a genus two, smooth, geometrically integral, geometrically indecomposable, projective curve over $\KK_0$, with $V(\KK_0) = \varnothing$, which is not defined over any finite field. In particular, we can form Frobenius twists of $V$ which we denote by $V^{(p^n)}$ for $n \in \ZZ$. Note that all $V^{(p^n)}$'s are also genus two, smooth, geometrically integral, projective curves over $\KK_0$ as base-changes of such. Moreover, these Frobenius twists can be assumed to be pairwise non-isomorphic (even geometrically). Indeed, to ensure this we can pick $V$ to be defined by the Weierstrass equation for genus two curves, outside countably many closed conditions on parameters defining $V$.

\begin{lemma}~\label{lemma_characteristic_p_Riemann_Hurwitz}
    Assume that $C$ is a genus two, smooth, geometrically integral, projective curve over a field $\KK$ extending $\KK_0$ with a non-constant map to $V_{\KK}$ over $\KK$. Then $C$ is isomorphic to some negative Frobenius twist of $V_{\KK}$, possibly after passing to the algebraic closure of $\KK$.
\end{lemma}
\begin{proof}
    Fix a non-constant map $f:C \to V_{\KK}$ over $\KK$ from the assumptions. By \cite[Proposition 0CD2]{stacks-project} we can decompose it as $C \to C^{(p^n)} \to V_{\KK}$ for some natural $n$, where $C^{(p^n)} \to V_{\KK}$ induces a separable field extension. 

    Note that $C^{(p^n)}$ is also a genus two, smooth, geometrically integral, projective curve over a field $\KK$, by being a flat base-change of such. Hence, by the Riemann-Hurwitz theorem, as in \cite[Lemma 0C1D]{stacks-project}, the map $C^{(p^n)} \to V_{\KK}$ is an isomorphism.

    By definition $C^{(p^n)}$ and $ V_{\KK}$ are base-changes via $n$'th power of Frobenius of $C$ and $V_{\KK}^{(p^{-n})}$ respectively. Passing to the algebraic closure of $\KK$, the Frobenius map becomes an automorphism of $\overline{\KK}$. Hence, we get that over $\overline{\KK}$ the curve $C_{\overline{\KK}}$ is isomorphic to $V_{\overline{\KK}}^{(p^{-n})}$, which finishes the proof.
\end{proof}

\begin{proposition}~\label{proposition_counterexample}
    In the above context, the theory $T_V$ does not admit a model companion.
\end{proposition}
\begin{proof}
    Assume for a contradiction that the model companion $V\XF$ exists. Let $\cC \to S$ be the Weierstrass family of genus two smooth, geometrically integral, projective curves. Then each geometric isomorphism type of a fiber appears only finitely many times. For more details, see for example \cite{moduli_genus_two}.
    
    Consider a model $\KK \models V\XF$. Note that for $s \in S(\KK)$ the curve $C = \cC_s$ can be in one of two groups:
    \begin{itemize}
        \item either $C$ admits a non-constant map to $V_{\KK}$ over $\KK$, in which case $C(\KK)$ is empty,
        \item or, $C$ does not admit such a map, in which case $C(\KK)$ is infinite.
    \end{itemize}
    By geometric indecomposability of $V$, one could show that $V\XF$ eliminates $\exists^\infty$ as in Theorem~\ref{theorem_properties_of_VXF}. Hence, the set of $s \in S(\KK)$ belonging to any of two groups above is definable. 

    We claim that the set $\{s \in S(\KK) : \cC_s(\KK) = \varnothing \}$ is in fact countable, which will give a contradiction by compactness (since $\KK$ is an arbitrary model of $V\XF$). This follows from Lemma~\ref{lemma_characteristic_p_Riemann_Hurwitz} and the fact that in $\cC \to S$ each geometric isomorphism type of a fiber appears only finitely many times.
\end{proof}

Next, we list some potential extensions and questions of immediate relevance.
\begin{question} 
\begin{enumerate}
    \item The results presented in Section~\ref{sec:indecomp} are established for geometrically indecomposable varieties $V$. However, the model theory of $V\XF$ remains rather mysterious when $V$ is decomposable. It is anticipated that $V\XF$ would exhibit more exotic behavior in this scenario.
    \item In the proposition above, we give an example of a variety $V$ in characteristic $p$, such that $T_V$ does not admit a model companion. What about such examples in characteristic zero? Does excludability of $V$ implies some form of ``rigidity" of $V$?
    \item What are the concrete models of $V\XF$? For example, is there $k\models V\XF$ such that $k\subset \overline{\mathbb{Q}}$? This seems hard and should be very much related to the notion of diophantine stability~\cite{dio_stab}.
    \item In the examples of $V\XF$ where we can determine the Galois group, the Galois group becomes $\omega$-free, it might be worthwhile finding more examples along this line where the absolute Galois group becomes more interesting. One potential candidate is to look at fields avoiding $V$ in the language where we expanded by a group action. As suggested by the machinery developed in~\cite{piotr_finite_grp}, it is believable that the model companion in this setting exists as well.
\end{enumerate}
\end{question}

\bibliographystyle{amsalpha}
\bibliography{ref}

\end{document}